\documentclass[letterpaper,12pt]{article}
\usepackage{
amsmath,epsfig, amscd, 
xspace, 
stmaryrd}
\usepackage[margin=2cm
]{geometry}
\usepackage[usenames]{color}
\usepackage[english]{babel}                   
\usepackage[latin1]{inputenc} 
\usepackage{amsmath,amssymb}
\usepackage{epsfig}
\usepackage[all]{xy}
\xyoption{2cell}
\usepackage{theorem}

\newcommand{\op}{^o}

\newdir{ )}{{}*!/-5pt/@^{(}}
\newcommand{\eg}{e.g.,\xspace}
\newcommand{\ie}{i.e.,\xspace}

\CompileMatrices

\definecolor{DarkBlue}{rgb}{0,0.06,0.25}

\newcommand{\Z}{{\mathbb Z}}
\newcommand{\Q}{{\mathbb Q}}

\newcommand{\hfp}{H{\mathbf F}_p}

\newcommand{\cC}{{\mathcal C}}

\newcommand{\Tor}{\operatorname{Tor}}

\newcommand{\THH}{\mathrm{T\hspace{-.5mm}H\hspace{-.5mm}H}}
\newcommand{\HH}{\mathrm{H\hspace{-.5mm}H}}

\newcommand{\Sym}{\mathrm{Sym}^\infty}

\newcommand{\fib}{\twoheadrightarrow}
\newcommand{\cof}{\rightarrowtail}

\newcommand{\wefib}{\overset\sim\fib}

\newcommand{\ess}{\mathbf S}

\theoremstyle{plain}
{
\theorembodyfont{\rmfamily}

\newtheorem{ex}[subsubsection]{Example}

}
\newtheorem{theo}[subsubsection]{Theorem}

\newtheorem{lemma}[subsubsection]{Lemma}
\newtheorem{cor}[subsubsection]{Corollary}
\newtheorem{remark}[subsubsection]{Remark}

\newenvironment{proof}{\par\noindent{\it Proof: }}{\qed\par}
\newcommand{\fin}{\mathrm{Fin}}

\def\qedbox{$\Box$}
\newbox\QedBox \setbox\QedBox\vbox{\hrule height 1ex width .618ex}
\def\qedbox{\copy\QedBox}

\makeatletter
\def\qed{%
   {%
      \unskip
      \nobreak \hfil
      \penalty 50               
      \hskip 3em                
      \null \nobreak \hfil
      \qedbox
      \parfillskip=\z@skip
      \finalhyphendemerits=\z@  
      \endgraf                  
   }}
\makeatother

\usepackage[ps2pdf]{hyperref}
\begin{document}
\title{Higher Hochschild homology is not a stable invariant}
\author{Bj\o rn Ian Dundas and Andrea Tenti}

\maketitle
\begin{quote}{\small
  {\textbf Abstract}
Higher Hochschild homology is the analog of the homology of spaces, where the context for the coefficients -- which usually is that of abelian groups -- is that of commutative algebras.  Two spaces that are equivalent after a suspension have the same homology.  We show that this is {\em not} the case for higher Hochschild homology, providing a counterexample to a behavior so far observed in stable homotopy theory.
}
\end{quote}


\label{sec:intro}
Higher Hochschild homology appears in many guises, algebraically in the form presented by Loday and Pirashvili (see \eg \cite{MR1755114}) and topologically/categorically in the form of {\em factorization homology} (see \eg \cite{MR2827826}, \cite{MR3040746}, \cite{MR3330249}, \cite{MR3431668}) and {\em smash powers} or {\em higher topological Hochschild homology} (see \cite{MR2729005}/\cite{MR2737802} and \cite{BDS}).  

Regardless of variant, higher Hochschild homology takes as input an algebra and a space.  Keeping the algebra fixed, we get a very nicely behaved functor from spaces that sends homotopy equivalent input to isomorphic output; you can even calculate it for CW-complexes by means of cell attachments.  It is reasonable to think of it as a multiplicative version of singular homology, where the algebra plays the role of the coefficients.

The calculations that have so far appeared have shown a pattern that fits with this idea: just as for singular homology, the output has depended only on the {\em stable} homotopy type of the input.  We elaborate a bit on this in Subsection~\ref{sec:stable} and give some illustrative examples.

However, Theorem~\ref{theo:counter} shows that this is not always the case: the exact behavior depends on the algebra.  Our example is fully algebraic and not very deep, but strong enough to shatter the hope of the more fanciful variants of higher Hochschild homology being stable invariants.  We comment on these consequences in Section~\ref{sec:ess-alg}.

\section{Homology}
\label{sec:hihoho}

Homology of spaces and higher Hochschild homology are examples of the same constructions, the difference is only that whereas ordinary homology has coefficients in an abelian group (or a spectrum), higher Hochschild homology has commutative algebras as input.

Let $\cC$ be a category with (chosen) finite coproducts $\coprod$ with initial object $k$. To aid the intuition, the reader may consider the cases of sets (with coproduct disjoint union and $k=\emptyset$), abelian groups (with the usual coproduct $\oplus$ and $k$ the trivial group $0$), or commutative rational algebras (with coproduct the tensor product $\otimes$ and $k=\Q$).
Then $\cC$ is what is called ``tensored over the category $\fin$ of finite sets''; there is a functor 
$$\fin\times\cC\to\cC,\qquad (S,A)\mapsto\coprod_SA,$$
where $S\mapsto\coprod_SA$ is uniquely characterized up to isomorphism by preserving coproducts
 and $\coprod_{\{1\}}A=A$.  In a picture, it is the ``straight line through $(\emptyset,k)$ and $(\{1\},A)$''.  Concretely, $\coprod_{\{1,\dots,n\}} A$ is a choice of an $n$-fold coproduct of $A$ with itself.

\subsubsection{Examples}
 \begin{enumerate}
 \item  When $\cC$ is the category of sets, then $\coprod_SX\cong S\times X$.
 \item When $\cC$ is the category of abelian groups, then  $\bigoplus_SM\cong\Z[S]\otimes M$, where $\Z[S]$ is the free abelian group on the set $S$.  This is an inspiration for the word ``tensored''.  We have refrained from using the generic categorical notation of a tensor since this clashes with our intended application, and have opted for allowing the coproduct to be visible.
 \item When $\cC$ is the category of commutative $k$-algebras over a commutative ring $k$, then  $\bigotimes_SA$ is the $|S|$-fold tensor (over $k$) of $A$ with itself. 
 \item When $\cC$  is the category of commutative $\ess$-algebras, then  $\bigwedge_SA$ is the $|S|$-fold smash product of $A$ with itself.  Similarly we get smash powers in the category of commutative $k$-algebras for any commutative $\ess$-algebra $k$.   For concreteness, orthogonal spectra is our chosen  model for the category of spectra with symmetric monoidal smash product over the sphere spectrum $\ess$, but other variants would work equally well.
 \end{enumerate}

 \begin{remark}
   Note that all the examples we are interested in have arbitrary colimits, so our constructions extend to arbitrary sets $S$.  In the interest of concreteness, and since our intended application only requires finite sets, we refrain from discussing this further; beyond saying that $\coprod_SA$ is given as the filtered colimit of the $\coprod_UA$ where $U$ varies over the finite subsets of $S$.
 \end{remark}

\subsubsection{Simplicial input}
\label{sec:simpin}

A simplicial object in a category $\cC$ is a functor $\Delta\op\to\cC$ from the opposite of the category $\Delta$ of
finite nonempty ordered sets
and order preserving functions.  
In what follows, {\em space} is short for simplicial set.  If $X$ and $Y$ are two pointed simplicial sets, then $X\vee Y$ is the disjoint union with basepoints identified.

The functoriality of $(S,X)\mapsto\coprod_SX$ shows that both entries can be extended ``degreewise'' to simplicial objects: if $S$ is a simplicial finite set and $A$ a simplicial object in $\cC$, then $([s],[t])\mapsto\coprod_{S_s}A_t$ is naturally a bisimplicial object in $\cC$ of which we take the diagonal: $\coprod_SA=\{[s]\mapsto\coprod_{S_s}A_s\}$. 

\subsubsection{Examples}
\begin{enumerate}
\item When $\cC$ is the category of sets and $S,X$ are simplicial (finite) sets, then $\coprod_SX$ is the product of simplicial sets $S\times X$.
\item When $\cC$ is the category of abelian groups, consider the case where $S$ a finite simplicial set and $M$ abelian group (considered as a discrete simplicial abelian group).  Then the Moore complex $C_*\left(\bigoplus_SM\right)$ (given as usual by $C_n\left(\bigoplus_SM\right)=\bigoplus_{S_n}M$ with boundary maps given by alternating sums of face maps) is nothing but the standard chain complex for calculating the homology $H_*(S;M)$ of $S$ with coefficients in $M$.
\item Consider the case where $\cC$ is the category of commutative $k$-algebras over a commutative ring $k$.  If $S^1$ is the standard simplicial circle and $A$ a commutative $k$-algebra, then  
$$C_*\left(\bigotimes_{S^1}A\right)\,=\,\left\{\dots \to
A^{\otimes n+1}\overset\partial\rightarrow A^{\otimes n}\overset\partial\rightarrow\dots \overset\partial\rightarrow A^{\otimes 2}\overset\partial\rightarrow A \right\}$$ is the standard complex for calculating the (non-derived version of) Hochschild homology $\HH^k_*(A)$.  More generally, for $n\geq 1$, $\bigotimes_{S^n}A$ gives what Pirashvili calls the higher Hochschild homology of $A$.
\item Consider the case where $\cC$ is the category of commutative $\ess$-algebras.  If $A$ is a commutative $\ess$-algebra, then  
$$\bigwedge_{S^1}A
$$ is (under flatness hypotheses) a model for calculating the topological Hochschild homology $\THH_*(A)$. 
\end{enumerate}

In many respects, the properties we are used to from ordinary homology extend.  It is generally true that when you extend a functor degreewise from sets to simplicial sets, then it will preserve simplicial homotopies.  In our examples the target categories have natural model structures and if the $A$ in question is suitably ``flat'' it is even true that $S\mapsto\coprod_SA$ sends weak equivalences to weak equivalences and cofibrations to cofibrations.
The fact that the functor preserves colimits then gives a sort of excision: if $S\subseteq X$ is a subcomplex and $f\colon S\to Y$ a map of simplicial sets, then 
$${\coprod}_{X\coprod_SY}A\cong\left({\coprod}_XA\right)\coprod_{\left(\coprod_SA\right)}\left({\coprod}_YA\right)$$ is the pushout along a cofibration and so a ``homotopy pushout'' and so in principle possible to compute from its pieces.

\subsection{A stable invariant?}
\label{sec:stable}

The point of this note is that one property of ordinary homology which fails to generalize. Ordinary homology is a {\em stable invariant}; in particular, if $M$ is an abelian group, $X$ and $Y$ are simplicial sets  and $\Sigma X\simeq \Sigma Y$ is a homotopy equivalence of suspensions, then the homologies of $X$ and $Y$ are isomorphic. 

For the case $\cC$ being the category of sets, the failure of stable invariance is no surprise: for a given simplicial set $X$ there is no reason for $S\mapsto S\times X$ to be a stable invariant: letting $X$ consist of a single point gives the identity functor $S\mapsto \coprod_S*\cong S$!  

However, for commutative $k$-algebras (for $k$ a commutative ring or $\ess$-algebra), series of computations had led many to believe that higher (topological) Hochschild homology should be a stable invariant.  Apart from the calculational evidence, the following two observation often led one to the wrong conclusion.

\begin{ex} For free symmetric algebras $A$, the functor $S\mapsto\bigotimes_SA$ {\bf is} a stable invariant.

  More concretely, let $k$ be a commutative ring and $E^k(X)$ the symmetric $k$-algebra on a simplicial set $X$, \ie $E^k(X)$ is the monoid algebra $k[\Sym X]$ on the free symmetric monoid $\Sym X= \coprod_{n\geq 0}X^{\times n}/\Sigma_n$ on $X$.   
In other words, $E^k$ is the left adjoint to the forgetful functor from commutative simplicial $k$-algebras to simplicial sets; $E^k$ is the (up to isomorphism) unique functor from simplicial sets to commutative simplicial $k$-algebras preserving colimits with $E^k(\mathbf 1)$ polynomial on one generator $k[t]$; $E(X)\cong \bigotimes_Xk[t]$.  
Hence, $$\bigotimes_SE^k(X)\cong E^k(S\times X)=k[\Sym(S\times X)].$$
Since $\Sigma (S\times X)\cong \left((\Sigma S)\times X\right)\coprod_{X\coprod X}S^0$ (contract $X$ times each of the vertex points of $\Sigma S$) we see that the stable type of $S\times X$ depends only on the stable type of $S$.  Since $H_*Y\cong\pi_*\Sym Y$ and since homology is a stable invariant, the homotopy type of $\Sym(S\times X)$, and hence $\bigotimes_SE^k(X)$,  depends only on the stable type of $S$.  

In effect, if there is an equivalence $\Sigma S\simeq\Sigma T$, then $\bigotimes_SE^k(X)\simeq \bigotimes_TE^k(X)$ (even if it is not induced by a map $S\to T$).
This stable invariance extends from free to to smooth algebras.  

However, even even though we functorially can replace any simplicial commutative $k$-algebra with an equivalent simplicial commutative $k$-algebra $A$ which is a free commutative $k$-algebra in every simplicial degree, the non-functoriality of the argument above denies us to the conclusion that $S\mapsto\bigotimes_SA$ is a stable invariant.
\end{ex}
\begin{ex}
  Let $f\colon X\to Y$ be a map of spaces such inducing a weak equivalence of suspensions $\Sigma f\colon\Sigma X\to\Sigma Y$, then a homotopy coend construction as in \cite[2.2.1.3]{DGM} shows that for commutative $k$-algebras ($k$ a commutative ring or commutative $\ess$-algebra), then the induced map $\bigotimes_XA\to\bigotimes_YA$ is an equivalence.  The crucial point is that in this case, the equivalence of suspensions is induced by a map of underlying (unsuspended) spaces.
\end{ex}

\section{The counterexample}
\label{sec:counterex}
As commented above, we know that $S\mapsto\bigotimes_SA$ is a stable invariant if $A$ is a smooth commutative $k$-algebra.  It turns out that this is sharp: stable invariance breaks down at the slightest singularity:
\begin{theo}\label{theo:counter}
  If $T^2=S^1\times S^1$ is the torus, then $\bigotimes_{T^2}\Q[t]/t^2$ and $\bigotimes_{S^1\vee S^1\vee S^2}\Q[t]/t^2$ are not equivalent simplicial abelian groups.
\end{theo}

This is a counterexample to stable invariance, since the torus and the wedge of spheres are equivalent after one suspension: $\Sigma(T^2)\simeq S^2\vee S^2\vee S^3$. 

There are three levels of sophistication when studying our counterexample:
\begin{enumerate}
\item $\bigotimes_{T^2}\Q[t]/t^2$ and $\bigotimes_{S^1\vee S^1\vee S^2}\Q[t]/t^2$ are not equivalent as  simplicial $\Q[t]/t^2$-algebras,
\item $\bigotimes_{T^2}\Q[t]/t^2$ and $\bigotimes_{S^1\vee S^1\vee S^2}\Q[t]/t^2$ are not equivalent as simplicial rings and
\item $\bigotimes_{T^2}\Q[t]/t^2$ and $\bigotimes_{S^1\vee S^1\vee S^2}\Q[t]/t^2$ are not equivalent as simplicial $\Q$-vector spaces.
\end{enumerate}
The meaning of {1.} is that (upon choosing basepoints for $T^2$ and $S^1\vee S^1\vee S^2$) both simplicial $\Q$-algebras have the structure of simplicial $\Q[t]/t^2$-algebras.  To study {1.} we can work with ``trivial coefficients'' - that is, we tensor both sides with $\Q$ over $\Q[t]/t^2$ and show that the results are different.  This simplifies things considerably since everything is flat over $\Q$ and we have K\"unneth formulae and the like.  Even so, it turns out that we need what amounts to a low-dimensional analysis of the attaching map $S^1\to S^1\vee S^1$ -- the commutator -- in order to get at the calculation for the torus.  For our singular ring $\Q[t]/t^2$, the attaching map is nontrivial, leading in Subsection~\ref{sec:modt} a full calculation of the case with trivial coefficients, verifying 1.

In Subsection~\ref{sec:diffmult} we undertake a slightly more elaborate low dimensional analysis, in particular of the products of certain generators in degree $1$, leading to a verification of 2.

Lastly, in Subsection~\ref{sec:fullcalc} we use a Bockstein sequence argument to show that that the calculations with trivial coefficients actually assemble to a full additive calculation of both sides, and the answers turn out to be different in degree $4$ and higher.

For $A$ a simplicial abelian group, we let $\pi_*A$ be the homotopy groups of the underlying pointed simplicial set ($\pi_*A$ is naturally isomorphic to the homology of the associated Moore complex $Ch_*(A)$).

\subsection{The Bockstein sequence}
\label{sec:Boc}
Let $B\fib C$ be a surjection of commutative simplicial rings with kernel $I$.  Then $\pi_*B\to\pi_*C$ is a map of graded commutative rings, $\pi_*I$ is a $\pi_*B$-modules and the boundary map $\partial\colon\pi_*C\to\pi_{*-1}I$ in the induced long exact sequence
$$\dots\to\pi_{n+1}C\to\pi_nI\to\pi_nB\to\pi_nC\to\pi_{n-1}I\to\dots\to\pi_0B\to\pi_0C\to 0$$
is a derivation.  We are interested in the following case; consider the projection $f\colon\Q[t]/t^2\to\Q$ with kernel $\Q\{t\}$.  
 For any space $X$ and $\Q[t]/t^2$-module $M$, let 
$$L_*(X;M)=\pi_*\left(M\otimes_{\Q[t]/t^2}\bigotimes_X\Q[t]/t^2\right)$$ (so that, when $X=S^1$, this is the usual Hochschild homology of $\Q[t]/t^2$ with coefficients in $M$), and abbreviate $L_*(X)=L_*(X;\Q[t]/t^2)$.  Then we have a ``Bockstein sequence''
$$\xymatrix{\dots\ar[r]& L_{n+1}(X;\Q)\ar[r]^{\partial}& L_{n}(X;\Q)\ar[r]^{j}& L_n(X)\ar[r]^{f}& L_n(X;\Q)\ar[r]&\dots\\
&&\dots\ar[r]^{j}& L_1(X)\ar[r]^{f} &L_1(X;\Q)\ar[r]& 0,}$$ 
where $f\colon L_*(X)\to L_*(X;\Q)$ is a map of graded commutative $\Q$-algebras, $j$ is a map of graded $L_*(X)$-modules and the boundary map $\partial\colon  L_{*+1}(X;\Q)\to L_{*}(X;\Q)$ is a derivation.  

\subsection{Calculating $L_*(S^n;\Q)$ and $L_*(S^n)$}
\label{sec:calcsph}
We use the following notation for free graded commutative $\Q$-algebras.  If $u_{2n}$ is an even dimensional class, $P(u_{2n})$ is the polynomial $\Q$-algebra generated by $u_{2n}$, and if $u_{2n-1}$ is an odd dimensional class, then $E(u_{2n-1})$ is the exterior $\Q$-algebra generated by $u_{2n-1}$.  Similarly with more variables.

Calculating directly with the shuffle product in the Hochschild complex \cite[4.2]{MR1600246}, we see that
$$L_*(S^1)\cong\Q\ltimes\Q\{x_0,x_1,x_2\dots\},$$
the square zero extension of $\Q$ by the vector space generated by $x_0, x_1,x_2,\dots$, with $x_i$ the class in degree $i$ represented by the cycle $t^{\otimes(i+1)}$ if $i$ is even and $1\otimes t^{\otimes i}$ if $i$ is odd (in particular, $t$ represents $x_0$). Likewise,
$$L_*(S^1;\Q)\cong E(y_1)\otimes P(y_2),$$
with $y_1$ being the class represented by $1\otimes t$ and $y_2$ the class represented by $1\otimes t\otimes t$.
Here, the nonzero values in the Bockstein sequence are $f(x_{2n+1})=y_1y_2^n$, $\partial(y_2^{n})=ny_1y_2^{n-1}$ and  $j(y_2^n)=x_{2n}$ for $n\geq 0$.

\begin{lemma}\label{lem:LSQ}
  There are isomorphisms 
$$L_*(S^{2n};\Q)\cong P(\sigma^{2n-1}y_1)\otimes E(\sigma^{2n-1}y_2)
\quad\text{ and }\quad
L_*(S^{2n+1};\Q)\cong E(\sigma^{2n}y_1)\otimes P(\sigma^{2n}y_2)$$ with $|\sigma^jy_i|=i+j$.
\end{lemma}
\begin{proof}
We use an inductive argument over the dimension $k$ of the sphere, assuming we have obtained the desired result for $L_*(S^k;\Q)$.  A $\Tor$-spectral sequence argument will give the result, or alternatively we may proceed as follows.
Consider the cofiber sequence $S^k \subset D^{k+1} \to S^{k+1}$ given by collapsing the boundary of the disk to one point. The Greenlees spectral sequence \cite{DLRI}  then takes the form:
$$
E^2_{*,*} = L_*(S^{k+1},\Q) \otimes L_*(S^k,\Q) \Rightarrow L_{*+*}(D^{k+1},\Q)=\Q.$$
Since all positive classes have to die, $\sigma^{k-1}y_1\in L_k(S^k;\Q)$ must be the target of a differential: $d^{k+1}(\sigma^ky_1)=\sigma^{k-1}y_1$ for a class $\sigma^ky_1\in L_{k+1}(S^{k+1};\Q)$.  Likewise, $\sigma^{k-1}y_2$ has to be hit by a $d^{k+2}$-differential, giving rise to $\sigma^ky_2\in L_{k+2}(S^{k+1};\Q)$. 
Considering the multiplicative structure of the spectral sequence, we see that $L_*(S^{k+1};\Q)$ must have the desired form. 
\end{proof}

\begin{cor}\label{cor:Lvee}
  There are isomorphisms 
$$L_*(S^1\vee S^1;\Q)\cong L_*(S^1;\Q)\otimes L_*(S^1;\Q)=E(y_1^h,y_1^v)\otimes P(y_2^h,y_2^v)$$ and 
$$L_*(S^1\vee S^1\vee S^2;\Q)\cong L_*(S^1\vee S^1;\Q)\otimes  L_*(S^1;\Q)=E(y_1^h,y_1^v,\sigma y_2)\otimes P(y_2^h,y_2^v,\sigma y_1),$$ 
with $|y^?_i|=i$ and $|\sigma y_i|=i+1$, $i=1,2$.
\end{cor}

\begin{lemma}
  There are isomorphisms $$L_*(S^{2n})\cong P_{\Q[t]/t^2}(\sigma^{2n-1}x_1)/t\sigma^{2n-1}x_1\otimes E_{\Q[t]/t^2}(\sigma^{2n-1}x_2)/t\sigma^{2n-1}x_2,$$
where $|\sigma^jx_i|=i+j$ and such that $f(\sigma^{2n-1}x_1)=\sigma^{2n-1}y_1$ and $j(\sigma^{2n-1}y_2)=\sigma^{2n-1}x_2$ in the Bockstein sequence; and
$$L_*(S^{2n-1})\cong \Q\ltimes\Q\{x_0,x_{2n-1},x_{2n},x_{4n-1},x_{4n},\dots\}$$
(square zero extension with $x_0$ corresponding to $t$),
where $|x_k|=k$, $f(x_{2kn-1})=\sigma^{2n-2}y_1(\sigma^{2n-2}y_2)^{k-1}$ and $j(y_{2n}^k)=x_{2kn}$.
\end{lemma}
\begin{proof}
  First: $L_i(S^k)=0$ for $0<i<k$ and $L_k(S^k)=\Q\{\sigma^{k-1}x_1\}$ \cite{MR1755114} where $f(\sigma^{k-1}x_1)=\sigma^{k-1}y_1\in L_k(S^k;\Q)$.  Hence $\partial\colon L_{k+1}(S^k;\Q)\to L_k(S^k;\Q)$ is bijective: $\partial\sigma^{k-1}y_2$ is a nonzero multiple of $\sigma^{k-1}y_1$.  Define $\sigma^{k-1}x_2=j(\sigma^{k-1}y_2)$.  The result then follows from Lemma~\ref{lem:LSQ} and the fact that, in the Bockstein sequence, $f$ and $j$ are appropriatly multiplicative and $\partial$ is a derivation.
\end{proof}

  Additively, the heart of the analysis above is that $(L_*(S^k;\Q),\partial)$ is a complex (in the sense that $\partial^2=0$) with homology a single copy of $\Q$ concentrated in degree zero.  This implies that $\dim_\Q L_s(S^k)=\dim_\Q L_s(S^k;\Q)$ for $s>0$, or, more compactly, the Hilbert (Poincar\'e) series (over $\Q$) satisfy $P_{L_*(S^{k};\Q)}(x)= P_{L_*(S^{k})}(x)-1$, with 
$$P_{L_*(S^{2n};\Q)}(x)= 
\frac{1+x^{2n+1}}{1-x^{2n}},\qquad
P_{L_*(S^{2n-1};\Q)}(x)= 
\frac{1+x^{2n-1}}{1-x^{2n}}.$$
Furthermore, the K\"unneth formula as applied to $(L_*(-;\Q),\partial)$ implies that the Hilbert series for $S^1\vee S^1\vee S^2$ satisfies the formula:
\begin{lemma}\label{dimLX}
$$P_{L_*(S^1\vee S^1\vee S^2;\Q)}(x)=P_{L_*(S^1\vee S^1\vee S^2)}(x)-1=\frac{1+x^3}{(1-x)^2(1-x^2)}=\frac{1-x+x^2}{(1-x)^3},$$
so that $\dim_\Q L_n(S^1\vee S^1\vee S^2)=\frac{n^2+n+2}{2}$ for $n>0$.
\end{lemma}

The multiplicative structure requires some more care.

\subsection{Low dimensional calculations}
\label{sec:low}

Viewing $T^2=S^1\times S^1$ as the diagonal of the bisimplicial set $\{([s],[t])\mapsto S^1_s\times S^1_t\}$ results in bicomplexes with $(\Q[t]/t^2)^{\otimes (s+1)(t+1)}$ and $\Q\otimes(\Q[t]/t^2)^{\otimes ((s+1)(t+1)-1)}$ in bidegree $(s,t)$ calculating 
$L_i(T^2)$ and $L_i(T^2,\Q)$.  We write elements of bidegree $(s,t)$ as $(t+1)\times(s+1$)-matrices (suppressing the tensor symbols, so that $
\begin{smallmatrix}
  a_{00}&a_{10}&a_{20}\\a_{01}&a_{11}&a_{21}
\end{smallmatrix}$ is an element in bidegree $(2,1)$ with the $a_{00}$ in the slot of the bimodule which is either $\Q[t]/t^2$, $t\Q[t]/t^2$ or $\Q$). The vertical and horizontal boundary maps entering the zeroth row or column are all zero (by commutativity), and so the classes in degree $1$ are represented by the horizontal/vertical cycles $1\, t$ and $
\begin{smallmatrix}
  1\\t
\end{smallmatrix}$
in bidegree $(1,0)$ and $(0,1)$: 
$$L_1(T^2)=\Q\{x^h_1,x^v_1\}\underset{\cong}{\overset{f}{\longrightarrow}}L_1(T^2;\Q)=\Q\{y^h_1,y^v_1\},\qquad f(x^h_1)=y^h_1,f(x^v_1)=y^v_1.
$$
Likewise, 
$$L_2(T^2)=\Q\{x^h_2,x^v_2,\sigma x_1,t\sigma x_1\}{\overset{f}{\longrightarrow}}L_2(T^2;\Q)=\Q\{y^h_2,y^v_2,\sigma y_1\},\quad f(x^h_2)=f(x^v_2)=0, f(\sigma x_1)=\sigma y_1,
$$
where $x^h_2$ and $x^v_2$ are represented by $t\,t\,t$ and $
\begin{smallmatrix}
  t\\t\\t
\end{smallmatrix}$,
while  $y^h_2$ and $y^v_2$ are represented by $1\,t\,t$ and $
\begin{smallmatrix}
  1\\t\\t
\end{smallmatrix}$, whereas $\sigma x_1$ and $\sigma y_1$ both are represented by
$
\begin{smallmatrix}
  1&1\\1&t
\end{smallmatrix}
$ (with the difference that $t\sigma y_1=0$).  Comparison through the inclusion of/projections to the coordinate circles show that $\partial(y^?_2)=y^?_1$.  Furthermore, the projection $T^2\to S^2$ collapsing the $1$-skeleton induces maps that send the classes $\sigma x_1$ and $\sigma y_1$ to classes of the same name.

\subsection{Mod $t$ calculations}
\label{sec:modt}
\begin{lemma}\label{lem:LTQ}
  There is an isomorphism $L_*(T^2;\Q)\cong P(\sigma y_1)\otimes E(y_1^h,y_1^v)/y_1^hy_1^v\otimes P(y_2^h,y_2^v).$
\end{lemma}
\begin{proof}
  Consider the cofiber sequence $S^1\vee S^1\subseteq T^2\to S^2$ given by collapsing the $1$-skeleton $S^1\vee S^1$ of the torus $T^2$.  This gives rise to a Greenlees spectral sequence of the form
$$E^2_{s,t}=L_s(S^2,\Q)\otimes L_t(S^1\vee S^1;\Q)\Rightarrow L_{s+t}(T^2;\Q).$$
Since the dimension of the $E^2$-page in total degree $2$ is $4$, whereas $\dim L_2(T^2;\Q)=3$, we must have a nontrivial $d^3$-differential $E^3_{3,0}=L_3(S^2;\Q)\to E^3_{0,2}=L_2(S^1\vee S^1;\Q)$.  However, from comparison with coordinate spheres, we have that $y^h_2,y^v_2\in E^2_{2,0}$ are permanent cycles, so $d^3\sigma y_2$ must be a nonzero multiple of $y_1^hy_1^v$.

Consequently, $$E^4=P(\sigma y_1)\otimes E(y_1^h,y_1^v)/y_1^hy_1^v\otimes P(y_2^h,y_2^v),$$
and the multiplicative structure leaves no room for further differentials or extensions.
\end{proof}

Hence, the Hilbert series of $L_*(T^2;\Q)$ is 
$$P_{L_*(T^2;\Q)}(x)=\frac{1+2x}{(1-x^2)^3},$$ 
that is to say, $\dim_\Q L_{2n}(T^2;\Q)=\frac{(n+1)(n+2)}{2}$ and $\dim_\Q L_{2n+1}(T^2;\Q)= (n+1)(n+2)$.
Comparing with Lemma~\ref{cor:Lvee}, we see that there are no isomorphisms between $L_*(T^2;\Q)$ and $L_*(S^1\vee S^1\vee S^2;\Q)$; establishing the first step of our proof of Theorem~\ref{theo:counter}.

\subsection{Difference in multiplicative structure}
\label{sec:diffmult}

We now show that the ring structures of $L_*(T^2)$ to $L_*(S^1\vee S^1\vee S^2)$ are different, thus finishing the second part of the proposed proof of the counterexample displayed in Theorem~\ref{theo:counter}.

An isomorphism of graded rings from $L_*(T^2)$ to $L_*(S^1\vee S^1\vee S^2)$ must send multiples of $t\in L_0(T^2)$ to multiples of $t\in L_0(S^1\vee S^1\vee S^2)$.

Both $L_1(T^2)$ and $L_1(S^1\vee S^1\vee S^2)$ (resp. $L_1(T^2;\Q)$ and $L_1(S^1\vee S^1\vee S^2;\Q)$) are generated by classes named $x_1^h$ and $x_1^v$ (resp. $y_1^h$ and $y_1^v$) coming from the inclusion of coordinate circles, and in the Bockstein sequence $f(x^?_1)=y^?_1$.  

Hence, $f(x^h_1x^v_1)=y^h_1y^v_1$ in both cases, but $y^h_1y^v_1=0$ in $L_2(T^2;\Q)$ while $y^h_1y^v_1\neq 0$ in $L_2(S^1\vee S^1\vee S^2;\Q)$.  This means that all classes in the image of the multiplication $L_1(T^2)\otimes L_1(T^2)\to L_2(T^2)$ are divisible by $t$, while the product $x^h_1x^v_1\in L_2(S^1\vee S^1\vee S^2;\Q)$ is {\bf not} divisible by $t$.  

Hence, there is no isomorphism of graded rings from $L_*(T^2)$ to $L_*(S^1\vee S^1\vee S^2)$.

\subsection{Difference in additive structure}
\label{sec:fullcalc}
With a little more work, we can see that $L_*(T^2)$ and $L_*(S^1\vee S^1\vee S^2)$ are different as graded vector spaces.  First consider the analysis of the low dimensional behavior in subsection~\ref{sec:low}.  In the Bockstein sequence for $L_*(T^2)$ we have that $\partial y_1^?=0$, $\partial y_2^?=y_1^?$, $\partial\sigma y_1=0$, $f(x_1^?)=y_1^?$, $f(y_2^?)=0$, $f(\sigma x_1)=\sigma y_1$, $j(y_1^?)=0$, $j(y_2^?)=x_2^?$ and $j(\sigma y_1)=t\sigma x_1$ (with $?=h,v$).  

By Lemma~\ref{lem:LTQ}, any product of an odd dimensional class with an odd dimensional class is zero ($y^h_1$ and $y^v_1$ are the only odd dimensional generators, and their product is zero).
This implies that for any $n$, the boundary map $\partial\colon L_{2n+1}(T^2;\Q)\to L_{2n}(T^2;\Q)$ is trivial.  On the other hand, the rank of $\partial\colon L_{2n}(T^2;\Q)\to L_{2n-1}(T^2;\Q)$ is $2+3+\dots+(n+1)=\frac{(n+1)(n+2)}{2}-1=\frac{n(n+3)}{2}$ (split according to the divisibility of $\sigma y_1$).

A quick calculation in the Bockstein sequence for $L_*(T^2)$ then yields that
$$\dim_\Q L_{2n-1}(T^2)=\frac{n(3n+1)}{2},\qquad \dim_\Q L_{2n}(T^2)=\frac{n^2+3n+4}2.
$$
In particular, $\dim_\Q L_4(T^2)=7$.  Comparing this with the calculation $\dim_\Q L_4(S^1\vee S^1\vee S^2)=11$ in \ref{dimLX}, we see that there is no isomorphism of graded vector spaces between $L_*(T^2)$ and $L_*(S^1\vee S^2\vee S^2)$.  This concludes the proof of Theorem~\ref{theo:counter}.

The authors admit they were somewhat baffled by the fact that for $j\leq 3$ the rational dimensions of $L_j(T^2)$ and $L_j(S^1\vee S^1\vee S^2)$ are equal.

\section{Commutative $\ess$-algebras}

\label{sec:ess-alg}
Let $A$ be a cofibrant replacement of the Eilenberg-Mac~Lane spectrum $H\Q[t]/t^2$ in the category of commutative $\ess$-algebras, $\ess\cof A\wefib H\Q[t]/t^2$.  Since the rationalization of $\ess$ is $H\Q$, the canonical map 
comparing smash and tensor
$$\bigwedge_XA\to H\left(\bigotimes_X\Q[t]/t^2\right)$$ is a weak equivalence 
for any space $X$.

Hence, Theorem~\ref{theo:counter} shows that $X\mapsto\bigwedge_XA$ is not a stable invariant.

\begin{remark}
  At present, we know of no examples showing that $X\mapsto\bigwedge_X\hfp$ is not a stable equivalence.  As a matter of fact, the calculation \cite{AD} showing that the equivariant structure of the smash power over the $n$-torus $\bigwedge_{T^n}\hfp$ detects the periodic map $v_{n-1}$ relies on a calculation using the splitting $\Sigma T^n\simeq \bigvee_{i=1}^n{n\choose{i}}S^{i+1}$.  
\end{remark}
\bibliographystyle{amsalpha}
\providecommand{\bysame}{\leavevmode\hbox to3em{\hrulefill}\thinspace}
\providecommand{\MR}{\relax\ifhmode\unskip\space\fi MR }
\providecommand{\MRhref}[2]{%
  \href{http://www.ams.org/mathscinet-getitem?mr=#1}{#2}
}
\providecommand{\href}[2]{#2}
\bibliographystyle{amsalpha}
\bibliography{dt.bib}
\end{document}